\documentclass{amsart}
\usepackage{amsmath,amsthm,amssymb}
\usepackage{graphicx}

\newtheorem{theorem}{Theorem}[section]
\newtheorem{proposition}[theorem]{Proposition}
\newtheorem{lemma}[theorem]{Lemma}

\usepackage{hyperref} 
\theoremstyle{definition}
\newtheorem{definition}[theorem]{Definition}
\newtheorem{remark}[theorem]{Remark}

\theoremstyle{theorem}

\title[Contact 5-manifolds admitting open books with exotic pages]{Contact 5-manifolds admitting open books\\with exotic pages}

\author[Selman Akbulut]{Selman Akbulut${}^1$}
\address{Department of Mathematics, \newline\indent Michigan State 
University \newline\indent East Lansing, MI 48824}
\email{akbulut@math.msu.edu}
\author[Kouichi Yasui]{Kouichi Yasui${}^2$}
\address{Department~of~Mathematics, Graduate School~of~Science, \newline\indent Hiroshima~University 1-3-1 Kagamiyama, \newline\indent  Higashi-Hiroshima, 739-8526, Japan}
\email{kyasui@hiroshima-u.ac.jp}
\thanks{\noindent ${}^1$Partially supported by NSF grants DMS 0905917 and  FRG 1065827.  
${}^2$Partially supported by JSPS KAKENHI Grant Number 25800048.}
\date{February 20, 2015}
\begin{document}
%\maketitle

\begin{abstract}
We construct a contact 5-manifold supported by infinitely many distinct open books with the identity monodromy and pairwise exotic Stein pages (i.e.\ pages are pairwise homeomorphic but non-diffeomorphic Stein fillings of a fixed contact 3-manifold), moreover we describe a process of generating infinitely many such examples. In contrast to this result, on each of $\#_{n} S^{2}\times S^{3}$ and $\#_{n}  S^{2}\tilde{\times} S^{3}$ $(n\geq 2)$ we construct infinitely many open books with pairwise exotic Stein pages (and identity monodromy) supporting mutually distinct contact structures.
\end{abstract}

\maketitle

%\vspace{-.2in}
\section{Introduction}\label{sec:intro}
For a given contact 3-manifold, let us consider the 5-dimensional open book whose page is a Stein filling of the contact 3-manifold and whose monodromy is the identity map. This open book supports a contact structure on the resulting closed 5-manifold, such contact structures were extensively studied,  for example  in \cite{DGV} Ding-Geiges-van Koert  classified a certain class of contact 5-manifolds. 
 
% \vspace{.05in}
\medskip
 
 From the view point of 4-dimensional symplectic topology, it is natural to ask how the choices of Stein fillings of a fixed contact 3-manifold affect the resulting contactomorphism types of $5$-manifolds. However, not much is known to the best of the authors' knowledge. Here we use the following terminology. 
\begin{definition}We say that a family of 5-dimensional open books has pairwise exotic Stein pages, if pages of its members are pairwise exotic (i.e.\ homeomorphic but non-diffeomorphic as smooth 4-manifolds), and are Stein fillings of the same contact 3-manifold. 
\end{definition}

Recently, Ozbagci-van Koert~\cite{OV} showed that infinitely many open books with exotic Stein pages and the identity monodromy can support pairwise distinct contact structures on the same smooth 5-manifold. %In fact, they found two such smooth 5-manifolds. 
They used the exotic Stein manifolds of \cite{AY6} as pages, and used  \cite{DGV} to distinguish their contact structures. %Therefore, it is interesting to see whether smoothly distinct Stein fillings of a fixed contact 3-manifold can induce the same contact 5-manifold. 
In this paper, we show that distinct open books with exotic Stein pages can support the same contact 5-manifold, contrary to the expectation from the result of Ozbagci-van Koert. 

\begin{theorem}\label{intro:thm:main}
There exists a contact 5-manifold supported by infinitely many distinct open books with pairwise exotic Stein pages and the identity monodromy. In fact, any closed contact 5-manifold with $b_2\geq 2$, that is induced from the boundary of a subcritical Stein manifold without 1-handles, is such an example. 
\end{theorem}

Note that by \cite{C} any 6-dimensional subcritical Stein manifold is the product of $D^2$ and a 4-dimensional Stein manifold. Therefore, contact 5-manifolds in this theorem are the contact boundary of the product of $D^2$ and Stein manifolds without 1-handles. Such smooth 5-manifolds are always diffeomorphic to either $\#_n S^2\times S^3$ or $\#_n S^2\tilde{\times} S^3$ with $n\geq 2$, and each of them admits infinitely many contact structures (cf.\ \cite{DGV}, see also Remark~\ref{rem:subcritical}). %\vspace{.1in}
\medskip

In contrast to the above theorem, we also extend the aforementioned result of Ozbagci-van Koert. Let $S^2\tilde{\times} S^3$ denote the total space of the non-trivial $S^3$-bundle over $S^2$, then the following holds (they proved the $n=2$ case). %by constructing new Stein structures on exotic Stein fillings obtained in \cite{Y6}. 
%We remark that they found infinitely many contact structures on two smooth 5-manifolds satisfying the condition below. 
\begin{theorem}\label{intro:thm:extension of OV}
For a closed 5-manifold diffeomorphic to either $\#_n S^2\times S^3$ or $\#_n S^2\tilde{\times} S^3$ with $n\geq 2$, there exist infinitely many open books with pairwise exotic Stein pages and the identity monodromy which support pairwise distinct contact structures on the 5-manifold. 
\end{theorem}

%Note that a 5-manifold in this theorem is  with $n=b_2$, according to the classification of 5-manifolds by Barden~\cite{Ba}. 

Here we outline our proofs of these results. These exotic Stein pages are diffeomorphic to exotic Stein fillings obtained in \cite{AY6} and \cite{Y6}. To obtain the results, we construct new Stein structures (and hence boundary contact structures) on these 4-manifolds. We distinguish 5-dimensional contact structures using results of \cite{DGV}. Regarding Theorem~\ref{intro:thm:extension of OV}, we use the argument similar to \cite{OV}. The new ingredient is 4-dimensional exotic Stein fillings obtained in \cite{Y6}. Also, we use yet another new Stein structures on these exotic 4-manifolds. %To obtain the result and simplify the argument, we construct yet another new Stein structures on these 4-manifolds.
\medskip\\
%%%%%%%%%%%%%%%%%
\textbf{Acknowledgements}. This work was done during the second author's stay at Michigan State University. He would like to thank them for their hospitality. 

%%%%%%%%%%%%%%%%%%%%%%
%%%%%%%%%%%%%%%%%%%%%%%%%
\section{Exotic Stein fillings of contact 3-manifolds}In this section, we construct new Stein structures on certain exotic smooth 4-manifolds. For basics of handlebody, Stein structures and contact structures, the readers can consult \cite{A_book}, \cite{GS} and \cite{OS1}. 

We first recall the infinitely many pairwise exotic 4-manifolds which are Stein fillings of the same contact 3-manifolds. For integers $m_0, m_1, m_2, p$ satisfying $m_0\geq 2$, $m_1\geq 1$, $m_2\geq 1$, and $p\geq 1$, let $X^{(m)}$ and $X^{(m)}_{p}$ be the simply connected compact smooth 4-manifolds with $b_2=2$ given by the handlebody diagrams in Figure~\ref{fig:smooth_nuclei}, where we put $m=(m_0,m_1,m_2)$. The $m=(3,2,1)$ case of these manifolds was studied in \cite{AY6} using Stein handlebody, and the general case was subsequently discussed in \cite{Y6} using Lefschetz fibration. We remark that each $X^{(m)}_{p}$ is obtained from $X^{(m)}$ by a Luttinger surgery (i.e.\ a logarithmic transformation with the multiplicity one), as shown in \cite{AY6} and \cite{Y6}. 
\begin{figure}[h!]
\begin{center}
\includegraphics[width=4.0in]{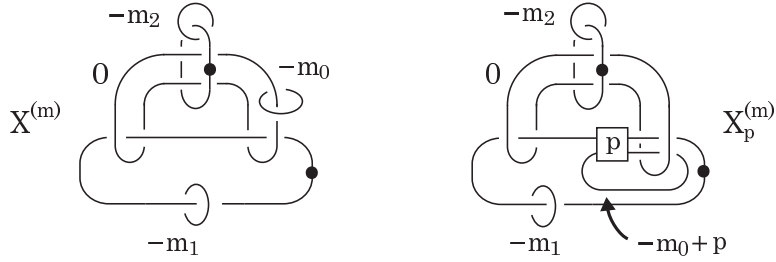}
\caption{$X^{(m)}$ and $X^{(m)}_p$ ($m_0\geq 2$,\, $m_1\geq 2$,\, $m_2\geq 1$,\, $p\geq 1$).}
\label{fig:smooth_nuclei}
\end{center}
\end{figure}

\vspace{.05in}
These 4-manifolds have the following properties, where we fix a given $3$-tuple $m=(m_0, m_1, m_2)$ of positive integers satisfying $m_0\geq 2$ and $m_1\geq 2$. 

\begin{theorem}[\cite{AY6}, \cite{Y6}]\label{thm:recall:nuclei} The infinite family $\{X_p^{(m)}\mid p\in \mathbb{N}\}$ contains infinitely many pairwise homeomorphic but non-diffeomorphic smooth 4-manifolds which are Stein fillings of the same contact 3-manifold. 

\newpage

\end{theorem}
\begin{theorem}[\cite{Y6}]\label{thm:recall:boundary_sum}  For a given Stein filling $Y$ of a contact 3-manifold, the infinite family $\{X_p^{(m)}\natural Y\mid p\in \mathbb{N}\}$ contains infinitely many pairwise  homeomorphic but non-diffeomorphic smooth 4-manifolds which are Stein fillings of the same contact 3-manifold. 
\end{theorem}
\begin{remark}\label{remark:exotic_Stein} These theorems hold for any choices of Stein structures on $X_p^{(m)}$'s, since every closed 3-manifold admits at most finitely many Stein fillable contact structures up to contactomorphism (see Lemma~3.4 in \cite{AY6}). Note that the boundary contact structures in these theorems depend on the choices of Stein structures on these 4-manifolds. 
\end{remark}

Next we construct a new Stein structure on $X_p^{(m)}$ to prove Theorem~\ref{intro:thm:main}. (For Theorem~\ref{intro:thm:extension of OV}, we use yet another Stein structure.) By canceling 1-handles of $X_p^{(m)}$, we obtain Figure~\ref{fig:2-handle_nuclei}. From this picture, we obtain the Stein handlebody diagram of $X_p^{(m)}$ in Figure~\ref{fig:Stein_nuclei_odd},  in the case where $m_1$ is odd. Here the boxes in the Stein picture means the Legendrian version of left handed full twists shown in Figure~\ref{fig:Legendrian_left_twists}. Note that each framing in the diagram is one less than the Thurston-Bennequin number. We equip $X^{(m)}_p$ with the Stein structure given by this picture. One can check the lemma below. 

\begin{lemma}\label{lem:picture:rotation} The rotation numbers of the attaching circles of the Stein handlebody $X^{(m)}_p$ in Figure~\ref{fig:Stein_nuclei_odd} are $0$ and $m_0-2$. 
\end{lemma}

%Since every closed 3-manifold has at most finitely many Stein fillable contact structures up to contactomorphism (see Lemma 3.4 in \cite{AY6}), Theorems~\ref{thm:recall:nuclei} and \ref{thm:recall:boundary_sum} imply the following lemma. 
%\begin{lemma}\label{lem:stein:exotic} For a given fixed $3$-tuple $m=(m_0, m_1, m_2)$ of positive integers satisfying $m_0\geq 2$, $m_1\geq 3$ and $m_1\equiv 1$, the following hold with respect to the Stein structure on $X_p^{(m)}$ given by Figure~\ref{fig:Stein_nuclei_odd}. 
%
%$(1)$ The infinite family $\{X_p^{(m)}\mid p\in \mathbb{N}\}$ contains infinitely many pairwise exotic Stein fillings of the same contact 3-manifold. 
%
%$(2)$ Fix a given Stein filling $Y$ of a contact 3-manifold. Then the infinite family $\{X_p^{(m)}\natural Y\mid p\in \mathbb{N}\}$ contains infinitely many pairwise exotic Stein fillings of the same contact 3-manifold
%\end{lemma}
\begin{figure}[h!]
\begin{center}
\includegraphics[width=4.0in]{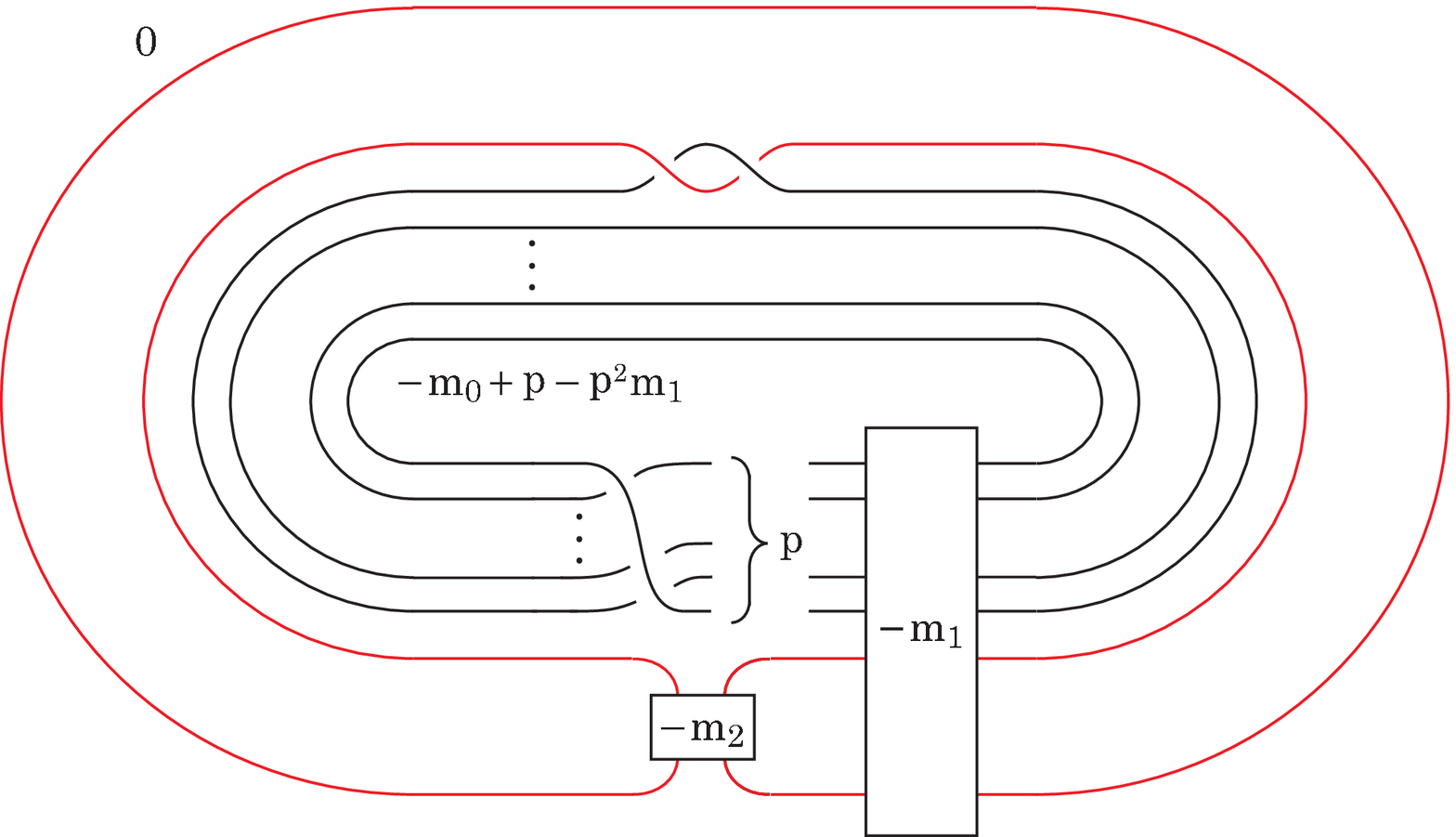}
\caption{$X^{(m)}_p$}
\label{fig:2-handle_nuclei}
\end{center}
\end{figure}
\begin{figure}[h!]
\begin{center}
\includegraphics[width=3.8in]{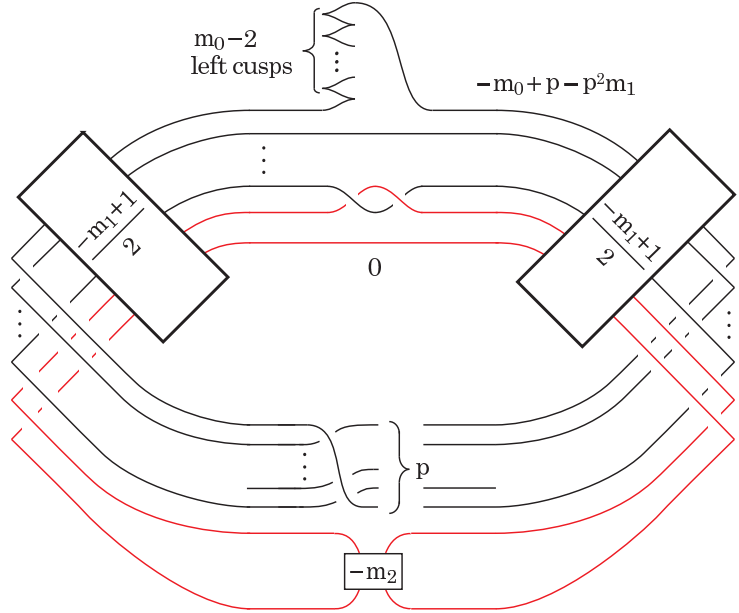}
\caption{Stein handlebody diagram of $X_p^{(m)}$ with $m_1\equiv 1\pmod{2}$}
\label{fig:Stein_nuclei_odd}
\end{center}
\end{figure}
\begin{figure}[h!]
\begin{center}
\includegraphics[width=4.0in]{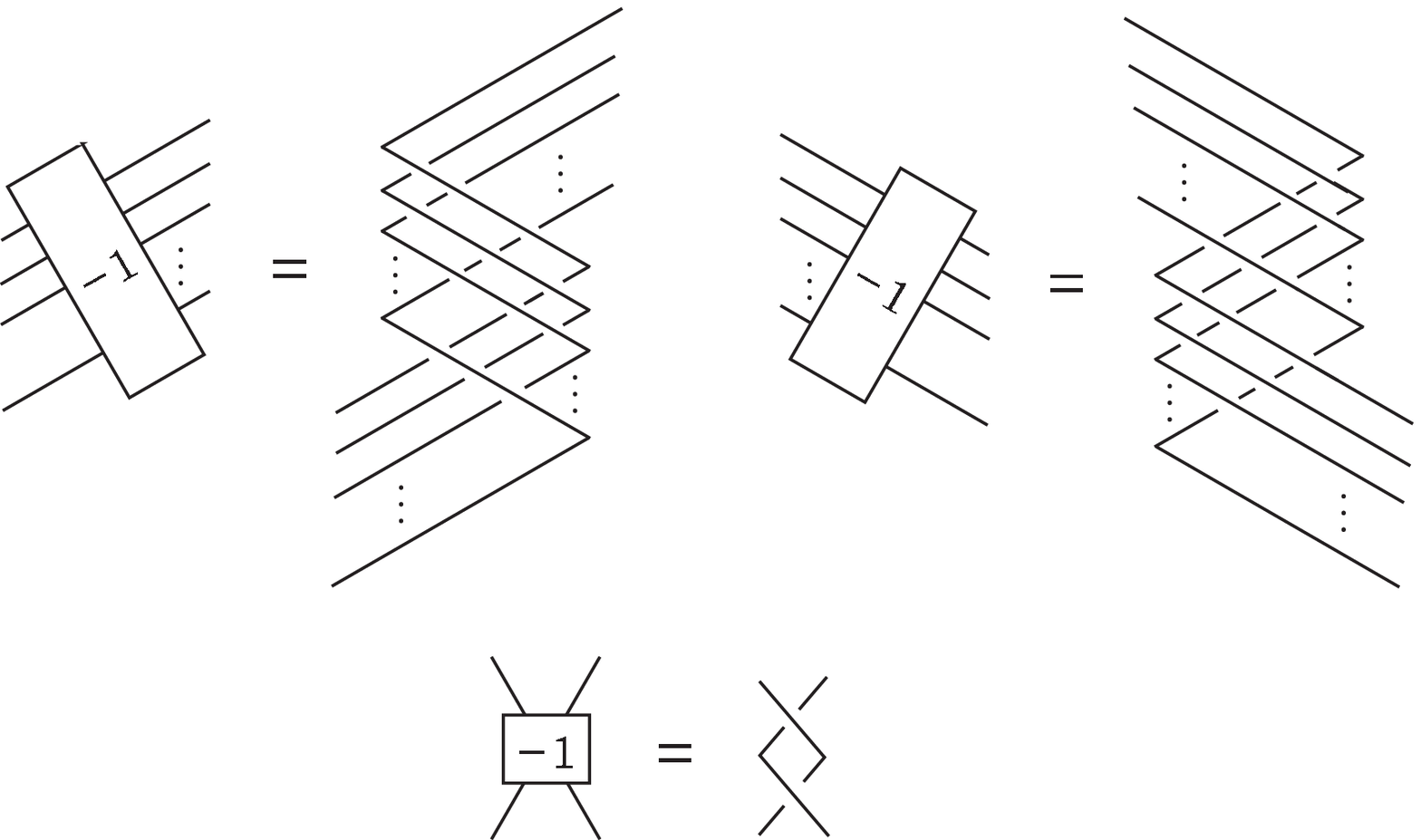}
\caption{Legendrian version of left handed full twists}
\label{fig:Legendrian_left_twists}
\end{center}
\end{figure}

%%%%%%%%%%%%%%%%%%%%%%%%%%
%%%%%%%%%%%%%%%%%%%%%%%%%%
%%%%%%%%%%%%%%%%%%%%%%%%%%%

\section{Contact 5-manifolds admitting open books with exotic pages}\label{sec:main}
In this section we prove Theorem~\ref{intro:thm:main} applying Stein fillings in the previous section. Throughout this section, we assume $m=(m_0,m_1,m_2)$ satisfies $m_1\equiv 1\pmod{2}$. %More precisely we prove the theorem below. 
%\begin{theorem}\label{thm:concrete_5-manifold}For each $n\geq 3$, every contact structure on both $\#_n S^2{\times} S^3$ and $\#_n S^2\tilde{\times} S^3$ admitting a subcritical Stein filling without 1-handles is supported by infnitely many distinct open books with pairwise exotic Stein pages and the identity monodromy.
%\end{theorem}
Let $(M^{(m)}_p, \xi_p^{(m)})$ be the contact 5-manifold supported by the open book $(X^{(m)}_p, id)$. For a positive integer $n$, let $Y_{n}$ be a 4-dimensional Stein handlebody with $b_2=n$ and without 1-handles such that the rotation number of the attaching circle of each 2-handle is $0$. One can easily find such a $Y_{n}$ for each $n$. We denote by $X^{(m)}_{p,n}$ the boundary connected sum $X^{(m)}_{p}\natural Y_{n}$, and equip it with the Stein structure induced from its natural Stein handlebody. Let $(M^{(m)}_{p,n}, \xi^{(m)}_{p,n})$ be the contact 5-manifold supported by the open book $(X^{(m)}_{p,n}, id)$. %As we will see later, $m,p,n,k$ are parameters for 5-dimensional contact structures, pages of open books, and smooth 5-manifolds, respectively. 
%We can easily check the diffeomorphism types of these 5-manifolds. 
%\begin{lemma}$(1)$ $M^{(m)}_p$ is diffeomorphic to $S^2 \times S^3\#S^2\times S^3$ (resp.\ $S^2 \tilde{\times} S^3\#S^2\tilde{\times} S^3$), if $-m_0+p-p^2m_1\equiv 0 \pmod{2}$ (resp.\ $-m_0+p-p^2m_1\equiv 1 \pmod{2}$)
%
%$(2)$ $N^{(m)}_{p,n,k}$ is diffeomorphic to $\sharp_{n+2}S^2 \times S^3$ (resp.\ $\sharp_{n+2}S^2 \tilde{\times} S^3$), if $-m_0+p-p^2m_1\equiv 0 \pmod{2}$ and $i=0$ ()
%\end{lemma}
First we recall the following useful results ($S^2\tilde{\times} S^3$ denotes the total space of the non-trivial $S^3$-bundle over $S^2$). 

\begin{theorem}[Ding-Geiges-van Koert, Proposition 4.5 in \cite{DGV}]\label{thm:DGV:knot}
Let $K_i$ $(i=1,2)$ be a Legendrian knot in the standard contact $S^3$, and let $(M_i,\xi_i)$ be the contact 5-manifold supported by the open book with the identity monodromy such that the page is the Stein filling obtained from $D^4$ by attaching a 2-handle along $K_i$ with contact $-1$ framing. Then $(M_1,\xi_1)$ and $(M_2,\xi_2)$ are contactomorphic to each other if and only if the absolute values of the rotation numbers of $K_1$ ad $K_2$ are eaqual to each other. Furthermore, $M_1$ is diffeomorphic to $S^2\times S^3$ $($resp.\ $S^2\tilde{\times} S^3$$)$, if the rotation number of $K_1$ is even $($resp.\ odd$)$. 
\end{theorem}

\begin{theorem}[Ding-Geiges-van Koert, Theorem 4.8 in \cite{DGV}]\label{thm:DGV}Let $(M_1,\xi_1)$ and $(M_2, \xi_2)$ be two simply connected closed contact 5-manifolds admitting subcritical Stein fillings without 1-handles. If there exists an isomorphism $H^2(M_1;\mathbb{Z})\to H^2(M_2;\mathbb{Z})$ that sends $c_1(\xi_1)$ to $c_1(\xi_2)$, then $(M_1,\xi_1)$ and $(M_2, \xi_2)$ are contactomorphic to each other. 
\end{theorem}

For simplicity of the notation, we use the following terminology. 
\begin{definition}
For a 4-dimensional Stein handlebody $X$ with $b_2\geq 1$ and without 1-handles, we call the maximal common divisor of the rotation numbers of the attaching circles of the 2-handles of $X$ as the rotation divisor of $X$, and denote it by $r(X)$. In the case where all the rotation numbers of the attaching circles are $0$, $r(X)$ is defined to be $0$. Note that each attaching circle is a Legendrian knot in the standard contact $S^3$. 
\end{definition}

We show the following proposition using the above theorems. 

\begin{proposition}\label{prop:DGV:rotation} Let $X_1, X_2$ be 4-dimensional Stein handlebodies with $b_2=n$ and without 1-handles. Then two contact 5-manifolds $(M_1, \xi_1)$ and $(M_2,\xi_2)$ supported by the open books $(X_1,id)$ and $(X_2,id)$ are contactomorphic to each other, if and only if $r(X_1)=r(X_2)$. Furthermore, $M_1$ is diffeomorphic to $\#_nS^2\times S^3$  $($resp.\ $\#_n S^2\tilde{\times} S^3$$)$, if $r(X_1)$ is even $($resp.\  odd$)$. 
\end{proposition}
\begin{proof}
 Due to the proof of Proposition 4.5 in \cite{DGV}, we see that there exists an isomorphism $H^2(M_1;\mathbb{Z})\to H^2(X_1;\mathbb{Z})$ which sends $c_1(\xi_1)$ to $c_1(X_1)$.  Therefore, according to Theorem~\ref{thm:DGV}, there exists an isomorphism $H^{2}(X_{1};\mathbb{Z})\to H^2(X_{2};\mathbb{Z})$ which maps $c_1(X_1)$ to $c_1(X_2)$, if and only if $(M_1, \xi_1)$ and $(M_2,\xi_2)$ are contactomorphic to each other. 
 
 A result of Gompf \cite{G1} tells that $c_1(X_1)$ is represented by a cocycle whose value on the 2-chain corresponding to each 2-handle of $X_1$ is the rotation number of the attaching circle of the 2-handle. Since $X_1$ is simply connected, we see that $H^2(X_1;\mathbb{Z})$ is isomorphic to $Hom(H_2(X_1;\mathbb{Z});\mathbb{Z})$. The maximal divisor of $c_1(X_1)$ is thus equal to $r(X_1)$. Therefore, there exists an isomorphism $H^{2}(X_{1};\mathbb{Z})\to H^2(X_{2};\mathbb{Z})$ which maps $c_1(X_1)$ to $c_1(X_2)$ if and only if $r(X_1)=r(X_2)$. Note that, for any two primitive elements of $\mathbb{Z}^n$, there exists an automorphism of $\mathbb{Z}^n$ which maps one to the other. The former claim of the proposition is now straightforward. 
 
 We check the latter claim of the proposition. Due to the former claim, we may assume that $X_1$ is the boundary sum of Stein handlebodies each of which is obtained from $D^4$ by attaching a 2-handle along a Legendrian unknot with the contact $-1$ framing, and that the parity of the rotation number of each Legendrian unknot coincides with the one of $r(X_1)$. Theorem~\ref{thm:DGV:knot} thus implies the latter claim of the proposition.
\end{proof}

As a consequence of this proposition, we see that rotation divisors determine supporting contact structures. We denote the resulting contact structure as follows. 
\begin{definition}\label{def:classification}For a 4-dimensional Stein handlebody with $b_2=n$ and without 1-handles, its rotation divisor $r$ uniquely determines a contact 5-manifold supported by the open book whose page is the given Stein handlebody and whose monodromy is the identity. We denote the resulting contact structure by $\zeta_{r,n}$. Note that the resulting 5-manifold is diffeomorphic to $\#_{n}S^2\times S^3$ (resp.\ $\#_{n}S^2\tilde{\times} S^3$), if $r$ is even (resp.\ odd). 
\end{definition}
\begin{remark}\label{rem:subcritical}
$(1)$ Due to the above proposition, $\zeta_{r,n}$ is contactomorphic to $\zeta_{r', n'}$ if and only if $r=r'$ and $n=n'$. \smallskip\\
$(2)$ Any closed contact 5-manifold admitting a subcritical Stein filling without 1-handles is contactomorphic to some $\zeta_{r,n}$. This is because a contact 5-manifold admits a subcritical Stein filling without 1-handles if and only if the contact structure admits an open book with the identity monodromy whose page is a Stein handlebody without 1-handle. In fact, such a 6-dimensional subcritical Stein filling is the product of $D^2$ and a 4-dimensional Stein filling without 1-handles. For these facts, see for example \cite{C}. Note that the boundary 5-manifold is simply connected, since it is a double of a 5-dimensional handlebody consisting of only 0- and 2-handles. 
%For each $n\geq 1$, the set $\{\zeta_{2d,n}\mid d\in \mathbb{Z}_{\geq 0}\}$ (resp.\ $\{\zeta_{2d+1,n}\mid d\in \mathbb{Z}_{\geq 0}\}$) gives all contact structures on $\#_{n}S^2\times S^3$ (resp.\ $\#_{n}S^2\tilde{\times} S^3$) admitting subcritical Stein fillings without 1-handles. 
\end{remark}

Now we can immediately determine the contactomorphism types of our contact 5-manifolds from Lemma~\ref{lem:picture:rotation} and Definition~\ref{def:classification}. 

\begin{proposition}\label{prop:contact_type_5-dim}

$(1)$ $(M^{(m)}_{p}, \xi_{p}^{(m)})$ is contactomorphic to $(\#_2 S^2\times S^3, \zeta_{m_0-2,2})$ if $m_0\equiv 0\pmod{2}$, and to $(\#_2 S^2\tilde{\times} S^3, \zeta_{m_0-2,2})$ if $m_0\equiv 1\pmod{2}$. \smallskip\\
$(2)$ $(M^{(m)}_{p,n}, \xi_{p,n}^{(m)})$ is contactomorphic to $(\#_{n+2} S^2\times S^3, \zeta_{m_0-2,n+2})$ if $m_0\equiv 0\pmod{2}$, and to $(\#_{n+2} S^2\tilde{\times} S^3, \zeta_{m_0-2,n+2})$ if $m_0\equiv 1\pmod{2}$. 
\end{proposition}

We are ready to prove Theorem~\ref{intro:thm:main}. 
\begin{proof}[Proof of Theorem~\ref{intro:thm:main}] Due to Remark~\ref{rem:subcritical}, it suffices to prove that, for each $n\geq 0$ and each $r\geq 0$, the contact structure $\zeta_{r,n+2}$ is supported by infinitely many distinct open books with pairwise exotic Stein pages and the identity monodromy. 

We assume $m=(m_0,m_1,m_2)$ satisfies $m_0\geq 2$, $m_1\geq 3$, $m_1\equiv 1 \pmod{2}$ and $m_2\geq 1$. For $n\geq 1$, Proposition~\ref{prop:contact_type_5-dim} tells that, each $(M^{(m)}_{p,n}, \xi_{p,n}^{(m)})$ is contactomorphic to $(\#_{n+2} S^2{\times} S^3, \zeta_{m_0-2,n+2})$ if $m_0$ is even, and to $(\#_{n+2} S^2\tilde{\times} S^3, \zeta_{m_0-2,n+2})$ if $m_0$ is odd. Therefore, varying $p$ and applying Theorem~\ref{thm:recall:boundary_sum}, we obtain the desired claim in the $n\geq 1$ case. Applying the same argument to $(M^{(m)}_{p}, \xi_{p}^{(m)})$, we obtain the $n=0$ case. 
\end{proof}

%As an immediate corollary, we obtain the following. 
%\begin{theorem}\label{thm:concrete_5-manifold}Assume that a contact structure on either $\#_n S^2{\times} S^3$ or $\#_n S^2\tilde{\times} S^3$ with $n\geq 2$ admits a subcritical Stein filling without 1-handles. Then the contact structure is supported by infnitely many distinct open books with pairwise exotic Stein pages and the identity monodromy.
%\end{theorem}

%%%%%%%%%%%%%%%%%%%%
%%%%%%%%%%%%%%%%%%%%

\section{Open books with exotic pages supporting distinct contact structures}
In this section we prove Theorem~\ref{intro:thm:extension of OV} by constructing a new Stein structure on $X_p^{(m)}$. In the rest of this section, we use the same symbol as those in the previous sections. However, we give different Stein or contact structures for the smooth manifolds $X_{p}^{(m)}$, $X_{p,n}^{(m)}$, $M_p^{(m)}$ and $M_{p,n}^{(m)}$. 

For a positive integer $p$ and a 3-tuple $m=(m_0,m_1,m_2)$ of positive integers with $m_0\geq 2$, we obtain the Stein handlebody diagram of $X_p^{(m)}$ in Figure~\ref{fig:many_contact} using Figure~\ref{fig:2-handle_nuclei}. We equip $X_p^{(m)}$ with the Stein structure given by this Stein handlebody diagram. Note that we do not need the extra condition $m_1\equiv 1\pmod{2}$ unlike the previous section. One can check the rotation numbers. 
\begin{lemma}\label{lem:picture:rotation2} The rotation numbers of the attaching circles of the Stein handlebody $X^{(m)}_p$ in Figure~\ref{fig:many_contact} are $0$ and $r(p,m):=p(m_1-1)+m_0-2$. 
\end{lemma}

\begin{figure}[h!]
\begin{center}
\includegraphics[width=4.0in]{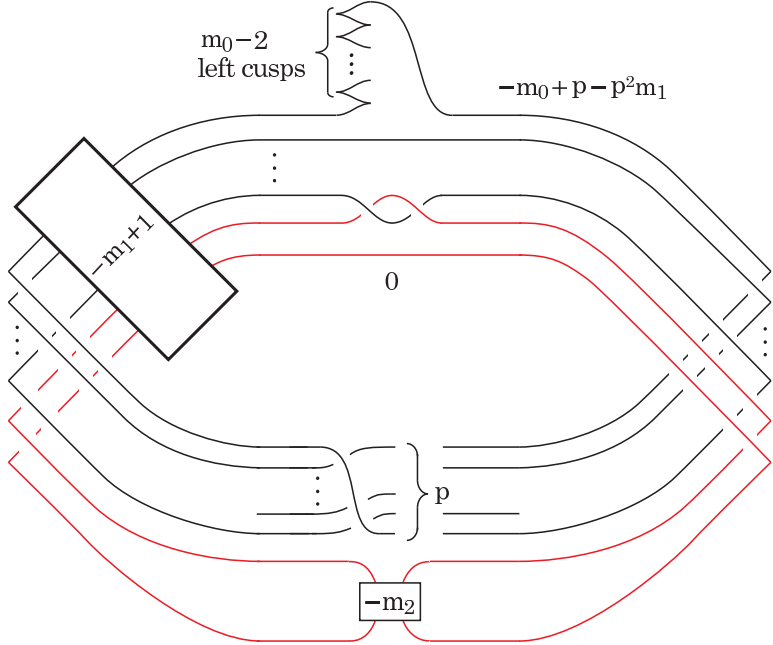}
\caption{Stein handlebody diagram of $X_p^{(m)}$}
\label{fig:many_contact}
\end{center}
\end{figure}

We define contact 5-manifolds $(M^{(m)}_p, \xi_p^{(m)})$ and $(M^{(m)}_{p,n}, \xi^{(m)}_{p,n})$ in the same way as Section~\ref{sec:main}, with respect to the above Stein structure on $X^{(m)}_p$. Similarly to the proof of Proposition~\ref{prop:contact_type_5-dim}, we can determine these contact structures. 

\begin{proposition}\label{prop:contact_type_many}
$(1)$ $(M^{(m)}_{p}, \xi_{p}^{(m)})$ is contactomorphic to $(\#_2 S^2\times S^3, \zeta_{r(p,m),2})$ if $r(p,m)\equiv 0\pmod{2}$, and to $(\#_2 S^2\tilde{\times} S^3, \zeta_{r(p,m),2})$ if $r(p,m)\equiv 1\pmod{2}$. \smallskip\\
$(2)$ $(M^{(m)}_{p,n}, \xi_{p,n}^{(m)})$ is contactomorphic to $(\#_{n+2} S^2\times S^3, \zeta_{r(p,m),n+2})$ if $r(p,m)\equiv 0\pmod{2}$, and to $(\#_{n+2} S^2\tilde{\times} S^3, \zeta_{r(p,m),n+2})$ if $r(p,m)\equiv 1\pmod{2}$. 
\end{proposition}

Now we can prove Theorem~\ref{intro:thm:extension of OV}. 

\begin{proof}[Proof of Theorem~\ref{intro:thm:extension of OV}]
We fix the parameter $m=(m_0,m_1,m_2)$ with $m_0\geq 2$ and $m_1\geq 2$. Proposition~\ref{prop:contact_type_many} tells that each $(M^{(m)}_{p,n}, \xi_{p,n}^{(m)})$ is contactomorphic to $(\#_{n+2} S^2\times S^3, \zeta_{r(p,m), n+2})$ if $r(p,m)$ is even, and to $(\#_{n+2} S^2\tilde{\times} S^3, \zeta_{r(p,m),n+2})$ if $r(p,m)$ is odd. Therefore, varying the parameter $p$ and using Theorem~\ref{thm:recall:boundary_sum} and Remark~\ref{rem:subcritical}, we see that both $\#_{n+2} S^2\times S^3$ and $\#_{n+2} S^2\tilde{\times} S^3$ satisfies the desired claim in the case $n\geq 1$. Applying the same argument to $(M^{(m)}_{p}, \xi_{p}^{(m)})$, we obtain the $n=0$ case. Therefore the theorem follows. 
%According to Barden~\cite{Ba}, if a simply connected closed 5-manifold with $b_2=n\geq 1$ and $\pi_3=0$ has no tosion in its second homology group, then it is diffeomorphic to either $\#_{n} S^2\times S^3$ or $\#_{n} S^2\tilde{\times} S^3$. 
\end{proof}

%As seen from the proof, Theorem~\ref{intro:thm:extension of OV} can be rewritten as follows. 
%\begin{theorem}\label{thm:concrete_5-manifold_many}For each $n\geq 2$, both $\#_n S^2{\times} S^3$ and $\#_n S^2\tilde{\times} S^3$ admit infinitely many distinct open books with pairwise exotic Stein pages and the identity monodromy such that these open books support pairwise non contactomorphic contact structures.\end{theorem}

\begin{remark}[On 6-dimensional subcritical Stein fillings] The proof of Theorem 4.8 in the paper \cite{DGV} of Ding-Geiges-van Koert together with the proof of our Theorem~\ref{intro:thm:main} tells that, for any fixed $n$ and $m$, our Stein handlebodies $X_{p,n}^{(m)}$ in Section~\ref{sec:main} are related to each other by certain moves of 2-handles (The same holds for $X_{p}^{(m)}$'s). According to the recent paper \cite{OV} of Ozbagci-van Koert, for 4-dimensional Stein handlebodies $X_1$ and $X_2$, the 6-dimensional subcritical Stein filling $X_1\times D^2$ is symplectically deformation equivalent to $X_2\times D^2$ with the contactomorphic boundary, if $X_1$ is related to $X_2$ by the above mentioned moves of handles. Hence, assuming their result, we see that infinitely many exotic 4-dimensional Stein fillings of a fixed contact 3-manifold can become the same 6-dimensional subcritical Stein filling (up to symplectic deformation) after taking the product with $D^2$. This seems to dash the hopes that smooth exotic structures of Stein $4$-manifolds might be detected by the complex structures of the $6$-manifolds obtained by the map $X\leadsto X\times D^{2}$.
\end{remark}
%%%%%%%%%%%%%%%%%%%%%%%%%%%%%%%%%%%%%%%%%%%%%%%%%%%%%%%%%
%%%%%%%%%%%%%%%%%%%%%%%%%%%%%%%%%%%%%%%%%%%%%%%%%%%%%inyouowari


\begin{thebibliography}{35}
%\bibitem{A1}S.\ Akbulut, \textit{A fake compact contractible $4$-manifold}, J.\ Differential Geom.\ \textbf{33} (1991), no.\ 2, 335--356.
%\bibitem{A6}S.\ Akbulut, \textit{A solution to a conjecture of Zeeman}, Topology, vol.30, no.3, (1991), 513--515. 
%\bibitem{A2}S.\ Akbulut, \textit{An exotic $4$-manifold}, J.\ Differential Geom.\ \textbf{33} (1991), no.\ 2, 357--361.
%\bibitem{A3}S.\ Akbulut, \textit{Constructing a fake $4$-manifold by Gluck construction  to a standard $4$-manifold,}, Topology, vol.27, no.\ 2 (1988), 239-243.

\bibitem{A_book}S.\ Akbulut, \textit{4-manifolds}, book in preparation, available at\\ 
 \url{http://www.math.msu.edu/~akbulut/papers/akbulut.lec.pdf}, (2014).

%\bibitem{AM}S.\ Akbulut and R.\ Matveyev, \textit{Exotic structures and adjunction inequality}, Turkish J.\ Math.\ \textbf{21} (1997), no.\ 1, 47--53.
%\bibitem{AM2}S.\ Akbulut and R.\ Matveyev, \textit{A convex decomposition theorem for $4$-manifolds}, Internat. Math. Res. Notices \textbf{1998}, no.\ 7, 371--381.
%\bibitem{AO1}S.\ Akbulut and B.\ Ozbagci, \textit{Lefschetz fibrations on compact Stein surfaces}, Geom.\ Topol.\ 5 (2001), 319--334.
%\bibitem{AO2}S.\ Akbulut and B.\ Ozbagci, \textit{Erratum: ``Lefschetz fibrations on compact Stein surfaces'' \textnormal{[Geom.\ Topol.\ 5 (2001), 319--334.]}}, Geom.\ Topol.\ 5 (2001), 939--945
%\bibitem{AO3}S.\ Akbulut and B.\ Ozbagci, \textit{On the topology of compact Stein surfaces}, Int.\ Math.\ Res.\ Not.\ 2002, no.\ 15, 769--782.

%\bibitem{AM}S.\ Akbulut and R.\ Matveyev, \textit {Exotic structures and adjunction inequality}, Turkish Jour. Math \textbf{21} (1997) 47-53.

%\bibitem{AY1}S.\ Akbulut and K.\ Yasui, \textit{Corks, Plugs and exotic structures}, J.\ G\"{o}kova Geom.\ Topol.\ GGT \textbf{2} (2008), 40--82.  
%\bibitem{AY3}S.\ Akbulut and K.\ Yasui, \textit{Knotting corks},  J.\ Topol.\ \textbf{2} (2009), no.\ 4, 823--839. 
%\bibitem{AY2}S.\ Akbulut and K.\ Yasui, \textit{Small exotic Stein manifolds}, Comment.\ Math.\ Helv.\ \textbf{85} (2010), no.\ 3, 705--721. 
%\bibitem{AY4}S.\ Akbulut and K.\ Yasui, \textit{Stein 4-manifolds and corks}, J.\ G\"{o}kova Geom.\ Topol.\ GGT \textbf{6} (2012), 58--79. 
%\bibitem{AY5}S.\ Akbulut and K.\ Yasui, \textit{Cork twisting exotic Stein 4-manifolds}, J.\ Differential Geom. \textbf{93} (2013), no. 1, 1--36. 
\bibitem{AY6}S.\ Akbulut and K.\ Yasui, \textit{Infinitely many small exotic Stein fillings}, arXiv:1208.1053, to appear in Journal of Symplectic Geometry. 

\bibitem{C} K. \ Cieliebak, \textit{ Subcritical Stein manifolds are split}, arXiv:math/0204351v1
%\bibitem{AY7}S.\ Akbulut and K.\ Yasui, \textit{Cork twisting exotic Stein 4-manifolds II}, in preparation. 

%\bibitem{AEMS}A.\ Akhmedov, J.\ B.\ Etnyre, T.\ E.\ Mark, and I. Smith, \textit{A note on Stein fillings of contact manifolds}, Math. Res. Lett. \textbf{15} (2008), no. 6, 1127--1132. 
%\bibitem{AkhOz1}A.\ Akhmedov and B.\ Ozbagci, \textit{Singularity links with exotic Stein fillings},  J.\ Singul.\ \textbf{8} (2014), 39--49. 
%\bibitem{AkhOz2}A.\ Akhmedov and B.\ Ozbagci, \textit{Exotic Stein fillings with arbitrary fundamental group}, arXiv:1212.1743v1. 
%\bibitem{Ba}D.\ Barden, \textit{Simply connected five-manifolds}, Ann.\ of Math.\ \textbf{82} (1965), 365--385.
%\bibitem{CGH}V.\ Colin, E.\ Giroux and K.\ Honda \textit{Finitude homotopique et isotopique des structures de contact tendues} (French) [Homotopy and isotopy finiteness of tight contact structures] Publ.\ Math.\ Inst.\ Hautes Etudes Sci.\ No.\ \textbf{109} (2009), 245--293. 
%\bibitem{C}C.\ L.\ Curtis, M.\ H. Freedman, W.\ C.\ Hsiang, and R.\ Stong, \textit{A decomposition theorem for $h$-cobordant smooth simply-connected compact $4$-manifolds}, Invent. Math. \textbf{123} (1996), no.\ 2, 343--348.
%\bibitem{DG1} F.\ Ding and H.\ Geiges, \textit{Symplectic fillability of tight contact structures on torus bundles}, Algebr.\ Geom.\ Topol.\ \textbf{1} (2001), 153--172.
%\bibitem{DG2} F.\ Ding and H.\ Geiges, \textit{A Legendrian surgery presentation of contact 3-manifolds}, Proc.\ Cambridge Philos.\ Soc.\ \textbf{136} (2004) 583--598.
%\bibitem{DG}F.\ Ding and H.\ Geiges, \textit{The diffeotopy group of $S^1\times S^2$ via contact topology}, Compos. Math. \textbf{146} (2010), no.\ 4, 1096--1112.
\bibitem{DGV}F.\ Ding, H.\ Geiges and O.\ van Koert, \textit{Diagrams for contact 5-manifolds}, J.\ Lond.\ Math.\ Soc.\ \textbf{86} (2012),  no.\ 3, 657--682.
\bibitem{E1}Y.\ Eliashberg. \textit {Topological characterization of Stein manifolds of dimension $>2$}, International J. of Math. Vol. 1 (1990), No 1  pp. 29-46.

%\bibitem{E2}Y.\ Eliashberg. \textit{Filling by holomorphic discs and its applications}, Geometry of low-dimensional manifolds, 2 (Durham, 1989), 45--67, London Math. Soc. Lecture Note Ser., 151, Cambridge Univ. Press, Cambridge, 1990. 
%\bibitem{ElP}Y.\ Eliashberg and L.\ Polterovich, \textit{New applications of Luttinger's surgery}, Comment.\ Math.\ Helv.\ \textbf{69} (1994), no. 4, 512-522.
%\bibitem{E3}Y.\ Eliashberg. \textit{A few remarks about symplectic filling}, Geom.\ Topol.\ \textbf{8} (2004) 277--293. 
%\bibitem{Et1}J.\ Etnyre, \textit{Introductory lectures on contact geometry}, Topology and geometry of manifolds (Athens, GA, 2001),  81--107, Proc.\ Sympos.\ Pure Math., \textbf{71}, Amer.\ Math.\ Soc., Providence, RI, 2003. 
%\bibitem{Et2}J.\ Etnyre, \textit{On symplectic fillings}, Algebr.\ Geom.\ Topol.\ \textbf{4} (2004),  73--80.
%\bibitem{Et4}J.\ Etnyre, \textit{Lectures on open book decompositions and contact structures}, Floer homology, gauge theory, and low-dimensional topology, 103--141, Clay Math.\ Proc., \textbf{5}, Amer.\ Math.\ Soc., Providence, RI, 2006.
%\bibitem{Et5}J.\ Etnyre, \textit{Contact geometry in low dimensional topology}, Low dimensional topology, 229--264, IAS/Park City Math.\ Ser., \textbf{15}, Amer.\ Math.\ Soc., Providence, RI, 2009. 
%\bibitem{EH}J.\ Etnyre and K.\ Honda, \textit{On symplectic cobordisms}, Math.\ Ann.\ \textbf{323} (2002) 31--39.
%\bibitem{Fr}M.\ Freedman, \textit{The topology of four-dimensional manifolds}, J.\ Differential Geom.\ \textbf{17} (1982), no.\ 3, 357--453.
\bibitem{G1}R.\,E.\ Gompf, \textit{Handlebody construction of Stein surfaces}, Ann. of Math.\ (2) \textbf{148} (1998), no. 2, 619--693.
%\bibitem{G_AGT}R.\ E.\ Gompf, \textit{More Cappell-Shaneson spheres are standard}, Algebr.\ Geom.\ Topol.\ \textbf{10} (2010),  no.\ 3, 1665--1681.
\bibitem{GS}R.\,E.\ Gompf and A.\,I.\ Stipsicz, \textit{$4$-manifolds and Kirby calculus}, Graduate Studies in Mathematics, \textbf{20}. American Mathematical Society, 1999.
%\bibitem{LM}P.\ Lisca and G.\ Mati\'c, \textit{Tight contact structures and Seiberg-Witten invariants}, Invent. Math. \textbf{129} (1997), no. 3, 509--525.
%\bibitem{KM1}P.\ Kronheimer and T.\ Mrowka, \textit{The genus of embedded surfaces in the projective plane}, Math.\ Res.\ Lett.\ \textbf{1} (1994), no.\ 6. 797--808.
%\bibitem{L}P.\ Lisca, \textit{Symplectic fillings and positive scalar curvature}, Geom.\ Topol.\ \textbf{2} (1998), 103--116.
%\bibitem{LM1}P.\ Lisca and G.\ Mati\'{c}, \textit{Tight contact structures and Seiberg-Witten invariants}, Invent.\ Math.\ \textbf{129} (1997) 509--525.
%\bibitem{LM2}P.\ Lisca and G.\ Mati\'{c}, \textit{Stein 4-manifolds with boundary and contact structures}, Topology and its Applications \textbf{88} (1998) 55--66.
%\bibitem{LS1}P.\ Lisca and A.\,I.\ Stipsicz, \textit{Ozsv\'{a}th-Szab\'{o} invariants and tight contact three-manifolds.\ I}, Geom.\ Topol.\ \textbf{8} (2004), 925--945. 
%\bibitem{LS2}P.\ Lisca and A.\,I. Stipsicz, \textit{Ozsv\'{a}th-Szab\'{o} invariants and tight contact three-manifolds.\ II}, J.\ Differential Geom.\ \textbf{75} (2007), 109--141.
%\bibitem{LoP}A.\ Loi and R.\ Piergallini, \textit{Compact Stein surfaces with boundary as branched covers of $B^4$}, Invent.\ Math.\ \textbf{143}  (2001),  no.\ 2, 325--348.
%\bibitem{M}R.\ Matveyev, \textit{A decomposition of smooth simply-connected $h$-cobordant $4$-manifolds}, J.\ Differential Geom.\ \textbf{44} (1996), no.\ 3, 571--582.
%\bibitem{Ma}B.\ Mazur, \textit{A note on some contractible $4$-manifolds}, 
%Ann. of Math.\ \textbf{73} (1961), 221--228.
%\bibitem{MMS}J.\ Morgan, T.\ Mrowka and Z.\ Szab\'{o}, \textit{Product formulas along $T^3$ for Seiberg-Witten invariants}, Math.\ Res.\ Lett.\ \textbf{4} (1997), 915--929.
%\bibitem{MST}J.\ Morgan, Z.\ Szab\'{o} and C.\ Taubes, \textit{A product formula for the Seiberg-Witten invariants and the generalized Thom conjecture}, J.\ Differential Geom.\ \textbf{44} (1996), 706--788.
\bibitem{OS1}B.\ Ozbagci and A I.\ Stipsicz, \textit{Surgery on contact 3-manifolds and Stein surfaces}, Bolyai Society Mathematical Studies, 13. Springer-Verlag, Berlin; Janos Bolyai Mathematical Society, Budapest, 2004.
%\bibitem{OS2}B.\ Ozbagci and A I.\ Stipsicz, \textit{Contact 3-manifolds with infinitely many Stein fillings}, Proc. Amer. Math. Soc. \textbf{132} (2004), no. 5, 1549--1558.

\bibitem{OV}B.\ Ozbagci and O. van Koert, \textit{Contact open books with exotic pages}, arXiv:1502.03233.  
 
%\bibitem{OzSz}P.\ Ozsv\'{a}th and Z.\ Szab\'{o}, \textit{The symplectic Thom conjecture}, Ann.\ of Math.\ \textbf{151} (2000), 93--124. 
%\bibitem{We}C.\ Wendl, \textit{Strongly fillable contact manifolds and $J$-holomorphic foliations}, Duke Math.\ J.\ \textbf{151} (2010), no.\ 3, 337--384. 
%\bibitem{We2}C.\ Wendl, personal communication.
%\bibitem{Wei}A.\ Weinstein, \textit{Contact surgery and symplectic handlebodies}, Hokkaido Math.\ J.\ \textbf{20} (1991), 241--251.
%\bibitem{Y5}K.\ Yasui, \textit{Nuclei and exotic 4-manifolds}, arXiv:1111.0620v2.
\bibitem{Y6}K.\ Yasui, \textit{Partial twists and exotic Stein fillings}, arXiv:1406.0050. 
\end{thebibliography}
\end{document}